\documentclass[12pt, a4]{amsart}
\usepackage{amsfonts}
\usepackage{amssymb}
\usepackage{amsmath}
\usepackage{amsthm}
\usepackage{fullpage}
\theoremstyle{plain}
\newtheorem{thm}{Theorem}[section]
\newtheorem{lem}[thm]{Lemma}
\theoremstyle{definition}

\newtheorem{ex}[thm]{Example}
\newtheorem{rem}[thm]{Remark}

\numberwithin{equation}{section}

\begin{document}

\title[Elliptic systems in exterior domains]{Nonzero radial solutions for elliptic systems with
coupled functional BCs in exterior domains}
\author[F. Cianciaruso]{Filomena Cianciaruso}%
\address{Filomena Cianciaruso, Dipartimento di Matematica e Informatica, Universit\`{a}
della Calabria, 87036 Arcavacata di Rende, Cosenza, Italy}
\email{cianciaruso@unical.it}
\author[G. Infante]{Gennaro Infante}%
\address{Gennaro Infante, Dipartimento di Matematica e Informatica, Universit\`{a} della
Calabria, 87036 Arcavacata di Rende, Cosenza, Italy}
\email{gennaro.infante@unical.it}
\author[P. Pietramala]{Paolamaria Pietramala}%
\address{Paolamaria Pietramala, Dipartimento di Matematica e Informatica, Universit\`{a}
della Calabria, 87036 Arcavacata di Rende, Cosenza, Italy}
\email{pietramala@unical.it}
\begin{abstract}
We prove new results on the existence, non-existence, localization and multiplicity of nontrivial radial solutions of a system of elliptic boundary value problems on exterior domains subject to nonlocal, nonlinear, functional boundary conditions. Our approach relies on fixed point index theory. As a by-product of our theory we provide an answer to an open question posed by do {\'O}, Lorca, S\'{a}nchez and Ubilla. We include some examples with explicit nonlinearities in order to illustrate our theory.
\end{abstract}
\maketitle
\section{Introduction}
The existence of solutions of elliptic systems on exterior domains has been studied by a number of authors. Two interesting papers in this direction are the ones by do \'O and co-authors~\cite{do4, do6}, where results on the existence, non-existence
and multiplicity of positive solutions of the elliptic system with \emph{non-homogenous} boundary conditions (BCs)
\begin{equation} \label{do}
\begin{cases}
\Delta u + f_1(|x|,u,v)=0,\ |x|\in [1,+\infty), \\
\Delta v +f_2(|x|, u,v)=0,\, |x|\in [1,+\infty),\\
u(x)=a\,\,\text{for  }x\in \partial B_{1},\, \,\,\,\displaystyle \lim_{|x|\to+\infty}u(|x|)=0, \\
v(x)=b\,\,\text{for  }x\in \partial B_{1},\,\,\,\, \displaystyle\lim_{|x|\to+\infty}v(|x|)=0,
\end{cases}
\end{equation}
were given.
The methodology in~\cite{do4, do6} relies on a careful use of the Krasnosel'ski\u\i{}-Guo Theorem on cone compressions and cone expansions
combined with the upper-lower solutions method and the fixed point index theory. These papers follows earlier works by do \'O et al~\cite{do0} on annular domains with non-homogenous BCs, and by Lee~\cite{lee1} on Dirichlet BCs on exterior domains.
In~\cite{do6} (see Open Problem 3) the authors posed an interesting question regarding the existence of multiple positive solutions of the elliptic system \eqref{do} under more general BCs.

Here we study the existence and the multiplicity of \emph{nonzero} solutions of the system of nonlinear elliptic BVPs with \emph{nonlocal} and \emph{functional}
BCs
\begin{equation}\label{ellbvp}
\begin{cases}
\Delta u + h_1(|x|) f_1(u,v)=0,\,\,\,\,\, |x|\in [R_1,+\infty), \\
\Delta v + h_2(|x|) f_2(u,v)=0,\,\,\,\,\, |x|\in [R_1,+\infty),\\
u(R_1x)=\beta_1 u(R_\eta x)\,\,\text{for  }x\in \partial B_{1},\ \,\, \,\, \displaystyle\lim_{|x|\rightarrow +\infty} u(|x|)=H_1[u,v],\\
v(R_1x)=\delta_1 \frac{\partial v}{\partial r}(R_\xi x)\,\,\text{for  }x\in \partial B_{1}, \,\, \,\,\displaystyle\lim_{|x|\to+\infty}v(|x|)=H_2[u,v],%
\end{cases}
\end{equation}
where $x\in \mathbb{R}^n $, $n\geq 3$, $\beta_1, \delta_1\in \mathbb{R}$, $R_1>0$, $R_\eta, R_\xi \in (R_1,+\infty)$, $B_\rho=\{ x\in\mathbb{R}^n : |x|<\rho\}$, $\dfrac{\partial}{\partial r}$ denotes (as in ~\cite{nirenberg}) differentiation in the radial direction $r=|x|$ and $H_{i}$ are suitable compact functionals, not necessarily linear.

We stress that a variety of methods has been used to study the existence of solutions of elliptic equations subject to \emph{homogeneous} BCs on exterior domains; for example topological methods where employed by Lee~\cite{lee1}, Stanczy~\cite{sta},  Han and Wang~\cite{han}, do \'O and others~\cite{do1}, Abebe and co-authors~\cite{abeshi} and Orpel~\cite{orpel}, a priori estimates were utilised by Castro et al~\cite{shi5},  sub and super solutions were used by Sankar and others~\cite{shi4} and Djedali and Orpel~\cite{orpel3} and variational methods were used by Orpel~\cite{orpel2}.

In the context of non-omogeneous BCs, elliptic problems in exterior domains were studied by Aftalion and Busca ~\cite{busca} and do \'O and others~\cite{do4, do2, do6, do7},  nonlinear BCs were investigated by Butler and others~\cite{but}, Dhanya et al~\cite{shi2}, Ko and co-authors ~\cite{shi3}, Lee and others~\cite{shi1}.

In order to discuss the existence of nonzero solutions of the elliptic system \eqref{ellbvp}, we study the associated system of perturbed
Hammerstein integral equations
\begin{equation*}\label{intrsyst}
\begin{cases}
u(t)=\left(1+\frac{(\beta_1-1)t}{1-\beta_1 \eta}\right) H_{1}[u,v]+\int_{0}^{1}k_1(t,s)g_1(s)f_1(u(s),v(s))\,ds, \\
v(t)=\left(1-\frac{t}{1-\beta_2 }\right) H_{2}[u,v]+\int_{0}^{1}k_2(t,s)g_2(s)f_2(u(s),v(s))\,ds.%
\end{cases}
\end{equation*}

The existence of solutions of systems of (different kinds) of perturbed
Hammerstein integral equations has been studied, for example, in \cite{df-gi-do, Goodrich1, Goodrich2, gifmpp-cnsns, gipp-ns, gipp-nonlin, gipp-mmas,  kang-wei, kejde, ya1}.
When using Krasnosel'ski\u\i{}-type arguments for these kind of systems, one difficulty to be overcome is how to control on the growth of the perturbation. In the recent manuscript~\cite{genupa} the authors used \emph{local} estimates via \emph{linear} functionals. Here we also use local estimates, but with \emph{affine} functionals instead. This allows more flexibility when dealing with the existence results.

The paper is organized as follows: Section~\ref{two} is devoted to the existence and non-existence results for the system~\eqref{ellbvp}; in Section~\ref{three} we briefly illustrate how our theory allows us to deal with more general conditions than the ones present in~\eqref{do}, giving a positive answer to the  Open Problem~3 of~\cite{do6}.

Our methodology relies on classical fixed point index theory (see for example \cite{Amann-rev, guolak}) and also benefit of ideas from the papers~\cite{df-gi-do, gifmpp-cnsns, gipp-nonlin, gi-pp-ft, gijwjiea, lan-lin-na, lanwebb, webb, jw-gi-jlms}.

\section{A system of elliptic PDEs in exterior domains}\label{two}
We consider the system of boundary value problems
\begin{equation}\label{ellbvpsec}
\begin{cases}
\Delta u + h_1(|x|) f_1(u,v)=0,\,\,\,\,\,|x|\in [R_1,+\infty), \\
\Delta v + h_2(|x|) f_2(u,v)=0,\,\,\,\,\,|x|\in [R_1,+\infty),\\
u(R_1x)=\beta_1 u(R_\eta x)\,\,\text{for  }x\in \partial B_{1},\ \, \,\, \displaystyle\lim_{|x|\rightarrow +\infty} u(|x|)=H_1[u,v],\\
v(R_1x)=\delta_1 \frac{\partial v}{\partial \sigma}(R_\xi x)\,\,\text{for  }x\in \partial B_{1}, \, \,\,\displaystyle\lim_{|x|\to+\infty}v(|x|)=H_2[u,v],%
\end{cases}
\end{equation}
where $x\in \mathbb{R}^n $, $\beta_1, \delta _1 \in \mathbb{R}$, $R_1>0$, $R_\eta, R_\xi \in (R_1,+\infty)$.
We assume that, for $i=1,2$,
\begin{itemize}
\item $f_i:\mathbb{R}\times \mathbb{R} \to[0,+\infty)$  is continuous;
\item  $h_i:[R_1,+\infty) \to [0,+\infty)$ is continuous  and $h_i(|x|)\leq \frac{1}{|x|^{n+\mu_i}}$ for $|x|\to +\infty$ for some $\mu_i>0$.
\end{itemize}

Consider  in $\mathbb{R}^n$, $n\ge 3$, the equation
\begin{equation}\label{eqell}
\triangle w+ h(|x|)f(w) = 0, \quad |x|\in
[R_{1},+\infty).
\end{equation}
In order to establish the existence of radial solutions $w=w(r)$, $r=|x|$, we proceed as in \cite{but} and rewrite \eqref{eqell} in the form
\begin{equation}\label{eqinterm}
w''(r) + \dfrac{n-1}{r}w'(r) + h(r)f(w(r))= 0,  \,\,\,\,  r\in [R_{1}, +\infty).
\end{equation}
Set $w(t)=w(r(t))$, where
\begin{equation*}
r(t):=R_1\,t^{\frac{1}{2-n}},\,\,\,\,\,\,t\in[0,1]
\end{equation*}
and take, for $t\in[0,1]$,
\begin{equation*}
\phi(t):=r(t)\,\frac{R_1}{(n-2)^2}\,t^{\frac{2n-3}{2-n}},
\end{equation*}
then~\eqref{eqinterm} becomes
\begin{equation*}
w''(t) + {\phi}(t) h(r(t)) f(w(t)) = 0,\,\,\,\,\,t\in[0,1].
\end{equation*}

Set
$u(t)=u(r(t))$ and $v(t)=v(r(t))$. Thus to the system \eqref{ellbvpsec} we associate the system of ODEs
\begin{equation}\label{1syst}
\begin{cases}
u''(t) + g_1(t) f_1(u(t),v(t)) = 0,\,\,\,\, t\in (0,1), \\
v''(t) + g_2(t) f_2(u(t),v(t)) = 0, \,\,\,\,t\in  (0,1),\\
u(0)=H_{1}[u,v],\,\,\, u(1)={\beta}_1u({\eta}),  \\
v(0)=H_{2}[u,v],\,\,\, v(1)=\beta_2 v'(\xi),
\end{cases}
\end{equation}
where
$$
g_i(t):={\phi}(t) h_i(r(t)),\,\,\,\,\,\,\,\,\,\,\beta_2=(2-n)\delta_1 \frac{\xi}{R_\xi}
$$
and $\xi, \eta \in (0,1)$ are such that $r(\eta)=R_{\eta}$ and $r(\xi)=R_{\xi}$.

We study the existence of nontrivial solutions of the system~\eqref{1syst}, by means of the associated system of perturbed Hammerstein integral equations
\begin{equation}\label{syst}
\begin{cases}
u(t)=\bigl(1+\frac{(\beta_1-1)t}{1-\beta_1 \eta}\bigr) H_{1}[u,v]+\int_{0}^{1}k_1(t,s)g_1(s)f_1(u(s),v(s))\,ds, \\
v(t)=\bigl(1-\frac{t}{1-\beta_2 }\bigr) H_{2}[u,v]+\int_{0}^{1}k_2(t,s)g_2(s)f_2(u(s),v(s))\,ds,%
\end{cases}
\end{equation}
where the Green's functions $k_1$ and $k_2$ are given by
\begin{equation}\label{k1}
k_1(t,s):=\dfrac{t}{1-\beta_1 \eta} (1-s)-
\begin{cases}
\dfrac{\beta_1 t}{1-\beta_1 \eta}(\eta -s), &  s \le \eta,\\
\quad 0,& s>\eta,
\end{cases}
 - \begin{cases}
  t-s, &s\le t, \\
  \quad 0,&s>t,
\end{cases}
\end{equation}
and
\begin{equation}\label{k2}
k_2(t,s):=\dfrac{t}{1-\beta_2}(1-s)-
\begin{cases}
\dfrac{\beta_2 t}{1-\beta_2}, &  s \le \xi,\\
\quad 0,& s>\xi,
\end{cases}
 - \begin{cases} t-s, &s\le t, \\
  \quad 0,&s>t.
\end{cases}
\end{equation}

Due to the presence of the parameters $\beta_1, \beta_2, \eta, \xi,$ several cases can occur;  here we restrict our attention to the case
\begin{equation}\label{param}
1\leq \beta_1<\frac{1}{\eta} \quad \mbox{and}\quad  0\le\beta_2<1-\xi.
\end{equation}
The choice~\eqref{param} is due to brevity and to the fact that it enables us to illustrate the different behaviour of the solution $(u,v)$ in the two components:
$u$ is \emph{non-negative} on $[0,1]$ and  $v$ is positive on some sub-interval of $ [0,1]$ and is allowed to \emph{change sign} elsewhere.

In the following Lemma we summarize some properties of the Green's functions $k_1$ and $k_2$ that can be found in~\cite{discont, webb}.
\begin{lem} The following hold:{}
\begin{itemize}
 \item[(1)]  The kernel $k_1$ is non-negative and continuous in  $[0,1]\times  [0,1]$. Moreover we have
\begin{align*}
0\leq & k_1(t,s) \leq \Phi_1(s),\ \text{ for } (t,s)\in [0,1] \times [0,1],\\
&k_1(t,s) \geq {c}_{k_1}\Phi_1(s),\ \text{ for } (t,s)\in [a_1,b_1] \times [0,1],
\end{align*}
where we take
$$
\Phi_1(s)= \frac{\beta_1s(1-s)}{1-\beta_1\eta},
$$
an arbitrary $[a_1,b_1]\subset(0,1]$ and
$${c}_{k_1}=\min\{a_1\eta, 4a_1(1-\beta_1\eta)\eta,\eta(1-\beta_1\eta)\}.$$
 \item[(2)]  The kernel $k_2$ in  $[0,1]\times  [0,1]$ is non-positive for
$$
1-\beta_2\leq t\leq1\,\,\,\,\ \text{and}\,\,\,\,\ 0\leq s \leq \xi
$$
and is discontinuous on the segment \{$t\in [0,1], s=\eta$\}. Moreover we have
\begin{align*}
& |k_2(t,s)|\leq \Phi_2(s),\ \text{ for } (t,s)\in [0,1] \times [0,1],\\
& k_2(t,s) \geq {c}_{k_2}\Phi_2(s)\ \text{ for } (t,s)\in [a_2,b_2] \times [0,1],
\end{align*}
where we take
$$
\Phi_2(s)=\max\{1,\beta_2/\xi\}\,\dfrac{s(1-s)}{1-\beta_2},
$$
an arbitrary $[a_2,b_2]\subset (0,1-\beta_2)$ and
$$
{c}_{k_2}=\dfrac{\min\{4a_2(1-\beta_2-\xi),(1-b_2-\beta_2)\}}{\max\{1,\beta_2/\xi\}}.
$$
\end{itemize}
\end{lem}

We note that, for $i=1,2$,
$$
g_i\,\Phi_i \in L^1[0,1], \,\,\,g_i \geq 0\,\,\,
$$
and assume that
\begin{itemize}
\item $\int_{a_i}^{b_i} \Phi_i(s)g_i(s)\,ds >0$.
\end{itemize}

We denote by $C[0,1]$ the space of the continuous functions on $[0,1]$ equipped with the norm $\|w\|:=\max\{|w(t)|,\; t\in [0,1]\}$ and set
$$
\gamma_1(t):=1+\frac{(\beta_1-1)t}{1-\beta_1 \eta}\,\,\,\,  \text{ and }\,\,\,\,\gamma_2(t):=1-\frac{t}{1-\beta_2 }.
$$
We have, for $i=1,2$, $\gamma_{i} \in C[0,1]$ and
\begin{equation*}
\gamma _{i}(t)\geq {c}_{\gamma_i}\| \gamma _{i}\|,\;\,\,\,\,\,t\in [a_{i},b_{i}],
\end{equation*}%
$$
||\gamma_{1}||=\frac{\beta_1(1-\eta)}{1-\beta_1 \eta},\,\,\,\,\,
{c}_{\gamma_1}=\frac{\beta_1-1}{\beta_1(1- \eta)}a_1+\frac{1-\beta_1\eta}{\beta_1(1-\eta)},
$$
$$||\gamma_{2}||=\begin{cases} \frac{\beta_2 }{1-\beta_2}, &  \beta_2 \geq 1/2,\\ \quad 1,& \beta_2\leq 1/2,
\end{cases}
\,\,\,\,\,\, \mbox{ and }\,\,\,\,\,\,{c}_{\gamma_2}=\begin{cases} \frac{1-\beta_2 }{\beta_2}-\frac{b_2}{\beta_2}, &  \beta_2 \geq 1/2,\\ \quad 1-\frac{b_2}{1-\beta_2},& \beta_2\leq 1/2.
\end{cases}
$$
We
consider the space $C[0,1]\times C[0,1]$ endowed  (with abuse of notation) with the norm $$\| (u,v)\| :=\max \{\| u\| ,\| v\| \}.$$
Recall that a \emph{cone} $K$ in a Banach space $X$  is a closed convex set such that $\lambda \, x\in K$ for $x \in K$ and
$\lambda\geq 0$ and $K\cap (-K)=\{0\}$.
Due to the properties above, we can work in the cone $K$ in $C[0,1]\times C[0,1]$ defined by
\begin{equation*}
\begin{array}{c}
K:=\{(u,v)\in {K_{1}}\times {K_{2}}\},%
\end{array}%
\end{equation*}
where
\begin{equation*}
K_{1}:=\{w\in C[0,1]:\,\,w\geq 0,\ \min_{t\in [a_{1},b_{1}]}w(t)\geq {c_{1}}\|
w\| \}
\end{equation*}%
and
\begin{equation*}
K_{2}:=\{w\in C[0,1]:\ \min_{t\in [a_{2},b_{2}]}w(t)\geq {c_{2}}\|
w\| \},
\end{equation*}%
with $c_{i}=\min \{{c}_{k_i},{c}_{\gamma_i}\}$.
Note that the functions in $K_{1}$ are non-negative, while in $K_{2}$ are allowed to change sign outside the interval $[a_{2},b_{2}]$.
$K_{1}$ is a kind of cone firstly used by Krasnosel'ski\u\i{}, see~\cite{krzab}, and D.~Guo, see e.g.~\cite{guolak}, while $K_{2}$
has been introduced by Infante and Webb in \cite{gijwjiea}.

By a \emph{nontrivial} solution of the system \eqref{ellbvpsec} we mean a solution $(u,v)\in K$ of the system \eqref{syst} such that $\|(u,v)\|>0$.

We assume that
\begin{itemize}
\item  For $i=1,2$, $H_{i}: K\to [0,\infty)$ is a compact functional.
\end{itemize}

Under our assumptions, it is possible to show that the integral operator
\begin{gather*}
\begin{aligned}    %\label{opT}
T(u,v)(t):=
\left(
\begin{array}{c}
\gamma_{1}(t)H_{1}[u,v]+F_1(u,v)(t) \\
\gamma_{2}(t)H_{2}[u,v]+F_2(u,v)(t)%
\end{array}
\right)
:=
\left(
\begin{array}{c}
T_1(u,v)(t) \\
T_2(u,v)(t)%
\end{array}
\right) ,
\end{aligned}
\end{gather*}
where
\begin{equation*}
F_i(u,v)(t):=\int_{0}^{1}k_i(t,s)g_i(s)f_i(u(s),v(s))\,ds,
\end{equation*}
leaves the cone $K$ invariant and is compact.
\begin{lem}
The operator $T$ maps $K$ into $K$ and is compact.
\end{lem}

\begin{proof}
Take $(u,v)\in K$ such that $||(u,v)||\leq r$. Then we have, for $t\in [0,1]$,
\begin{align*}
|T_{2}(u,v)(t)|=&\Bigl|\gamma _{2}(t)H _{2}[u,v]+\int_{0}^{1} k_2(t,s)g_{2}(s)f_{2}(u(s),v(s))\,ds\Bigr| \\
\leq& |\gamma _{2}(t)|H _{2}[u,v]+\int_{0}^{1} |k_2(t,s)|g_{2}(s)f_{2}(u(s),v(s))\,ds
\end{align*}
and, taking the supremum over $[0,1]$, we obtain
\begin{equation*}
||T_{2}(u,v)||\leq || \gamma _{2}|| H
_{2}[u,v]+\int_{0}^{1}\Phi
_{2}(s)g_{2}(s)f_{2}(u(s),v(s))\,ds.
\end{equation*}%
Moreover we have
\begin{eqnarray*}
\min_{t\in [a_{2},b_{2}]}T_{2}(u,v)(t) &\geq  {c}_{\gamma_2} || \gamma
_{2}|| H _{2}[u,v]+ {c}_{k_2} \int_{0}^{1}\Phi _{2}(s)g_{2}(s)f_{2}(u(s),v(s))\,ds \\
&\geq c_{2}||T_{2}(u,v)|| .
\end{eqnarray*}%
Hence we have $T_{2}(u,v)\in K_{2}$. In a similar manner we proceed for $T_{1}(u,v)$.

Moreover, the map $T$ is compact since the components $T_{i}$ are sum of two compact maps:  by routine arguments, the compactness of $F_{i}$ can be shown and, since $\gamma_{i}$  is continuous, the perturbation $\gamma_{i}(t)H _{i}[u,v]$ maps bounded sets into bounded subsets of a finite dimensional space.
\end{proof}
\subsection{Index calculations}
If $U$ is a open bounded subset of a cone $K$ (in the relative topology) we denote by $\overline{U}$ and $\partial U$ the closure and the boundary relative to $K$. When $U$ is an open bounded subset of $X$ we write $U_K=U \cap K$, an open subset of $K$.

We summarize in the next Lemma  some classical results regarding the fixed point index, for more details see~\cite{Amann-rev, guolak}.

\begin{lem}
Let $U$ be an open bounded set with $0\in U_{K}$ and $\overline{U}_{K}\ne K$. Assume that $S:\overline{U}_{K}\to K$ is a compact map such that $x\neq Sx$ for all $x\in \partial U_{K}$. Then the fixed point index $i_{K}(S, U_{K})$ has the following properties.
\begin{itemize}
\item[(1)] If there exists $e\in K\setminus \{0\}$ such that $x\neq Sx+\lambda e$ for all $x\in \partial U_K$ and all $\lambda>0$,
then $i_{K}(S, U_{K})=0$.
\item[(2)] If  $\mu x \neq Sx$ for all $x\in \partial U_K$ and for every $\mu \geq 1$, then $i_{K}(S, U_{K})=1$.
\item[(3)] If $i_K(S,U_K)\ne0$, then $S$ has a fixed point in $U_K$.
\item[(4)] Let $U^{1}$ be open in $X$ with $\overline{U_{K}^{1}}\subset U_K$. If $i_{K}(S, U_{K})=1$ and $i_{K}(S, U_{K}^{1})=0$, then $S$ has a fixed point in $U_{K}\setminus \overline{U_{K}^{1}}$. The same result holds if $i_{K}(S, U_{K})=0$ and $i_{K}(S, U_{K}^{1})=1$.
\end{itemize}
\end{lem}

For our index calculations we use the following (relative) open bounded sets in $K$:
\begin{equation*}
K_{\rho_1,\rho_2} = \{ (u,v) \in K : \|u\|< \rho_1\ \text{and}\ \|v\|< \rho_2\}
\end{equation*}
and
\begin{equation*}
V_{\rho_1,\rho_2} =\{(u,v) \in K: \min_{t\in [a_1,b_1]}u(t)<\rho_1\ \text{and}\ \min_{t\in
[a_2,b_2]}v(t)<\rho_2\}.
\end{equation*}
 The set $V_{\rho,\rho}$  was introduced in~\cite{gipp-ns} and is equal to the set called $\Omega^{\rho /c}$ in~\cite{df-gi-do}, an extension to the case of systems of a set given by Lan \cite{lan}. The choice of different radii (used also in the papers~\cite{chzh2, genupa, gipp-nodea}) allows more freedom in the growth of the nonlinearities.

 We utilize  the following Lemma, similar to Lemma $5$ of \cite{df-gi-do}. The proof follows as the corresponding one in~\cite{df-gi-do} and is omitted.

\begin{lem}  %\label{esca}
The sets $K_{\rho_1,\rho_2}$ and $V_{\rho_1,\rho_2}$ have the following properties:
\begin{enumerate}
\item $K_{\rho_1,\rho_2}\subset V_{\rho_1,\rho_2}\subset K_{\rho_1/c_1,\rho_2/c_2}$.
\item $(w_1,w_2) \in \partial V_{\rho_1,\rho_2}$ \; iff \; $(w_1,w_2)\in K$, $\displaystyle\min_{t\in [a_i,b_i]} w_i(t)= \rho_i$ for some $i\in \{1,2\}$ and $\displaystyle\min_{t\in [a_j,b_j]}w_j(t)\le \rho_j$ for $j\neq i$.
\item If $(w_1,w_2) \in \partial V_{\rho_1,\rho_2}$, then for some $i\in\{1,2\}$ $\rho_i \le w_i(t) \le \rho_i/c_i$ for $t \in [a_i,b_i]$ and for $j\neq i$ $0 \leq w_j(t) \leq \rho_j/c_j$ for  $t\in [a_j,b_j]$.
 \item $(w_1,w_2) \in \partial K_{\rho_1,\rho_2}$ \; iff \; $(w_1,w_2)\in K$, $\displaystyle\|w_i\|= \rho_i$ for some $i\in \{1,2\}$ and
 $\displaystyle\|w_j\| \le \rho_j$ for $j\neq i$.
\end{enumerate}
\end{lem}

Now we prove a result concerning the fixed point index on the set $K_{\rho_1,\rho_2}$.

\begin{lem}\label{ind1L}
Assume that
\begin{enumerate}
\item[$(\mathrm{I}_{\rho_1,\rho_2 }^{1})$]  there exist $\rho_1,\rho_2 >0$, $A^{\rho_1,\rho_2}_{1},A^{\rho_1,\rho_2}_{2}\geq 0$, linear functionals $\alpha^{\rho_1,\rho_2}_{ij}[\cdot]:K_{j}\rightarrow [0,+\infty)$, involving positive Stieltjes measures given by
$$
\alpha^{\rho_1,\rho_2}_{ij}[w]=\int_0^1 w(t)\,dC_{ij}(t),
$$
where $C_{ij}$ is of bounded variation, such that, for $i=1,2$,
\begin{itemize}
\item  $\alpha^{\rho_1,\rho_2}_{ii}[\gamma_{i}]<1$,
\item $H_{i}[u,v]\leq A_i^{\rho_1,\rho_2}+\alpha^{\rho_1,\rho_2}_{i1}[u]+\alpha^{\rho_1,\rho_2}_{i2}[v]$ for  $(u,v)\in \partial
K_{\rho_1,\rho_2}\,$,
\item the following inequality holds  with $j=1,2$, $j\not=i$:
\begin{equation}\label{ind1s}
 f_i^{\rho_1,\rho_2} \Bigl(\frac{ || \gamma _{i}||}{1 -\alpha_{ii}^{\rho_1,\rho_2}[\gamma_i]}\int_{0}^{1}\mathcal K_{i}(s)g_{i}(s)\,ds +\dfrac{1}{m_i}\Bigr)+||\gamma_i|| \dfrac{A_i^{\rho_1,\rho_2}+\rho_j\alpha_{ij}^{\rho_1,\rho_2}[1]}{\rho_i(1-\alpha_{ii}^{\rho_1,\rho_2}[\gamma_i])}<1,
\end{equation}{}
where
\begin{align*}
f_{i}^{\rho_1,\rho_2}:=&\sup \Bigl\{\frac{f_{i}(u,v)}{\rho_i}:\;(u,v)\in [ 0,\rho_1 ]\times [ -\rho_2,\rho_2 ]\Bigr\},\\
\mathcal K_{i}(s):=&\int_{0}^{1}k_i(t,s)\,dC_{ii}(t)\ \text{ and }\ \frac{1}{m_{i}}:=\sup_{t\in \lbrack 0,1]}\int_{0}^{1}|k_{i}(t,s)|g_{i}(s)\,ds.
\end{align*}
\end{itemize}
\end{enumerate}
Then $i_{K}(T,K_{\rho_1,\rho_2})$ is equal to 1.
\end{lem}

\begin{proof}
We show that $\mu (u,v)\neq T(u,v)$ for every $(u,v)\in \partial K_{\rho_1,\rho_2 }$ and for every $\mu \geq 1$.
In fact, if this does not happen, there exist $\mu \geq 1$ and $(u,v)\in\partial K_{\rho_1,\rho_2 }$ such that $\mu (u,v)=T(u,v)$. Firstly we assume that $\| u\| \leq\rho_1 $ and $\| v\| = \rho_2 $. Then we have, for $t\in[0,1]$,
\begin{equation}\label{29}
\mu v(t)=\gamma_2(t)H_2 [u,v]+F_{2}(u,v)(t).
\end{equation}
Applying $\alpha_{22}^{\rho_1,\rho_2}$ to both sides of \eqref{29}, we obtain
\begin{align*}
\mu \alpha_{22}^{\rho_1,\rho_2}[v]=&\alpha_{22}^{\rho_1,\rho_2}[\gamma_2]H_2[u,v]+\alpha_{22}^{\rho_1,\rho_2}[F_2(u,v)]\\
\le &\alpha_{22}^{\rho_1,\rho_2}[\gamma_2](A_2^{\rho_1,\rho_2}+\alpha_{21}^{\rho_1,\rho_2}[u]+\alpha_{22}^{\rho_1,\rho_2}[v])+\alpha_{22}^{\rho_1,\rho_2}[F_2(u,v)]\\
\le &\alpha_{22}^{\rho_1,\rho_2}[\gamma_2](A_2^{\rho_1,\rho_2}+\rho_1\alpha_{21}^{\rho_1,\rho_2}[1]+\alpha_{22}^{\rho_1,\rho_2}[v])+\int_{0}^{1}\mathcal K_{2}(s)g_{2}(s)f_2(u(s),v(s))\,ds.
\end{align*}
Therefore we obtain
\begin{align*}
(\mu-\alpha_{22}^{\rho_1,\rho_2}[\gamma_2])\alpha_{22}^{\rho_1,\rho_2}[v]
\le &\rho_2\alpha_{22}^{\rho_1,\rho_2}[\gamma_2]\Bigl(\frac{A_2^{\rho_1,\rho_2}}{\rho_2}
+\frac{\rho_1}{\rho_2}\alpha_{21}^{\rho_1,\rho_2}[1]\Bigr)\\
&+\rho_2\int_{0}^{1}\mathcal K_{2}(s)g_{2}(s)\frac{f_2(u(s),v(s))}{\rho_2}\,ds,
\end{align*}
and, consequently, we get
\begin{align*}
\alpha_{22}^{\rho_1,\rho_2}[v]
 &\le \frac{\rho_2\,\alpha_{22}^{\rho_1,\rho_2}[\gamma_2]}{\mu-\alpha_{22}^{\rho_1,\rho_2}[\gamma_2]}\Bigl(\frac{A_2^{\rho_1,\rho_2}}{\rho_2}
+\frac{\rho_1}{\rho_2}\alpha_{21}^{\rho_1,\rho_2}[1]\Bigr)\\
&\,\,\,\,\,\,\,\,+\frac{\rho_2}{\mu-\alpha_{22}^{\rho_1,\rho_2}[\gamma_2]}\int_{0}^{1}\mathcal K_{2}(s)g_{2}(s)\frac{f_2(u(s),v(s))}{\rho_2}\,ds\\
&\le \frac{\rho_2\,\alpha_{22}^{\rho_1,\rho_2}[\gamma_2]}{1-\alpha_{22}^{\rho_1,\rho_2}[\gamma_2]}\Bigl(\frac{A_2^{\rho_1,\rho_2}+\rho_1\alpha_{21}^{\rho_1,\rho_2}[1]}{\rho_2}\Bigr)\\
&\,\,\,\,\,\,\,\,+\frac{\rho_2}{1-\alpha_{22}^{\rho_1,\rho_2}[\gamma_2]}\int_{0}^{1}\mathcal K_{2}(s)g_{2}(s)\frac{f_2(u(s),v(s))}{\rho_2}\,ds.
\end{align*}
Moreover, we have
\begin{align*}
\mu |v(t)|&=|\gamma_2(t)H_{2} [u,v]+F_{2}(u,v)(t)|\\ & \leq |\gamma_{2}(t)|H_{2} [u,v]+\int_{0}^{1}|k_{2}(t,s)|g_{2}(s)f_2(u(s),v(s))\,ds\\
 & \leq |\gamma_{2}(t)|(A_2^{\rho_1,\rho_2}+\alpha_{21}^{\rho_1,\rho_2}[u]+\alpha_{22}^{\rho_1,\rho_2}[v])
+\int_{0}^{1}|k_{2}(t,s)|g_{2}(s)f_2(u(s),v(s))\,ds\\
& \leq \rho_2\Bigl(|\gamma_{2}(t)|\dfrac{A_2^{\rho_1,\rho_2}+\rho_1\alpha_{21}^{\rho_1,\rho_2}[1]}{\rho_2}+|\gamma_{2}(t)|\bigg(
\frac{\alpha_{22}^{\rho_1,\rho_2}[\gamma_2]}{1-\alpha_{22}^{\rho_1,\rho_2}[\gamma_2]}\Bigl(\dfrac{A_2^{\rho_1,\rho_2}+\rho_1\alpha_{21}^{\rho_1,\rho_2}[1]}{\rho_2}\Bigr)\\
&\,\,\,\,\,+\frac{1}{1-\alpha_{22}^{\rho_1,\rho_2}[\gamma_2]}\int_{0}^{1}\mathcal K_{2}(s)g_{2}(s)\frac{f_2(u(s),v(s))}{\rho_2}\,ds\bigg)
+\int_{0}^{1}|k_{2}(t,s)|g_{2}(s)\frac{f_2(u(s),v(s))}{\rho_2}ds\Bigr)\\
&=\rho_2\Bigl(|\gamma_{2}(t)|\dfrac{A_2^{\rho_1,\rho_2}+\rho_1\alpha_{21}^{\rho_1,\rho_2}[1]}{\rho_2(1-\alpha_{22}^{\rho_1,\rho_2}[\gamma_2])}+\frac{|\gamma_2(t)|}{1 -\alpha_{22}^{\rho_1,\rho_2}[\gamma_2]}\int_{0}^{1}\mathcal K_{2}(s)g_{2}(s)\frac{f_2(u(s),v(s))}{\rho_2}\,ds\\
&\,\,\,\,\,\,\,\,\,\,\,\,\,\,\,\,\,\,+\int_{0}^{1}|k_{2}(t,s)|g_{2}(s)\frac{f_2(u(s),v(s))}{\rho_2}ds\Bigr).
\end{align*}
Taking the supremum over $[0,1]$ gives
\begin{align*}
\mu {\rho_2} \leq \rho_2\Bigl(&||\gamma_2||\dfrac{A_2^{\rho_1,\rho_2}+\rho_1\alpha_{21}^{\rho_1,\rho_2}[1]}{\rho_2(1-\alpha_{22}^{\rho_1,\rho_2}[\gamma_2])}+\frac{||\gamma_2||}{1 -\alpha_{22}^{\rho_1,\rho_2}[\gamma_2]}f_2^{\rho_1,\rho_2}\int_{0}^{1}\mathcal K_{2}(s)g_{2}(s)\,ds\\
&+f_{2}^{\rho_1,\rho_2}\sup_{t\in [0,1]}\int_{0}^{1}|k_{2}(t,s)|g_{2}(s)ds\Bigr).
\end{align*}

Using the hypothesis \eqref{ind1s} we obtain $\mu \rho_2 <\rho_2$. This contradicts the fact that $\mu \geq 1$ and proves the result.

The case $\| u\| =\rho_1 $ and $\| v\| \leq \rho_2 $ is simpler and therefore is omitted.
\end{proof}

\begin{rem}
Take $\omega\in L^1([0,1]\times [0,1])$ and denote by
$$
\omega^+(t,s)=\max\{\omega(t,s),0\},\,\,\, \omega^-(t,s)=\max\{-\omega(t,s),0\}.
$$
Then we have
\begin{equation}\label{k+k-}
\left|\int_0^1 \omega(t,s)ds \right|\le\max\left\{\int_0^1\omega^+(t,s)ds,\int_0^1 \omega^-(t,s)ds \right\}\le \int_0^1|\omega(t,s)|ds,
\end{equation}
since $\omega=\omega^+ - \omega^-$ and  $|\omega|=\omega^++\omega^-$.

Note that, using the inequality~\eqref{k+k-} as in \cite{gipp-nodea, gi-pp-ft}, it is possible to relax the growth assumptions on the nonlinearity $f_2$ in Lemma~\ref{ind1L}, by replacing the quantity
$$
\sup_{t\in [0,1]} \int_0^1|k_2(t,s)| g_2(s)ds
$$
with
\begin{equation*}
\sup_{t \in [0,1]  }\left\{\max\left\{\int_0^1k_2^+(t,s)g_2(s)\,ds,\int_0^1k_2^-(t,s)g_2(s)\,ds\right\}\right\}.
\end{equation*}{}

For example, if we fix $\beta_2=1/2$, $\xi=1/3$ in \eqref{k2} and $g_2\equiv 1$, we obtain
\begin{equation*}
\sup_{t \in [0,1]  }\left\{\max\left\{\int_0^1k_2^+(t,s)\,ds,\int_0^1k_2^-(t,s)\,ds\right\}\right\}=\sup_{t \in [1/2,1] }{\frac{1}{18}(-9t^2+14t-1)}=0.247
\end{equation*}{}
and
$$
\sup_{t\in [0,1]} \int_0^1|k_2(t,s)| ds=\sup_{t \in [1/2,1] }{\frac{1}{18}(-9t^2+16t-2)}=0.284.
$$
\end{rem}

We give a first Lemma that shows that the index is 0 on a set $V_{\rho_1,\rho_2}$.

\begin{lem}\label{idx0n1}
Assume that
\begin{enumerate}
\item[$(\mathrm{I}_{\rho_1,\rho_2 }^{0})$]  there exist $\rho_1,\rho_2 >0$, $A^{\rho_1,\rho_2}_{1},A^{\rho_1,\rho_2}_{2}\geq 0$, linear functionals $\alpha^{\rho_1,\rho_2}_{ij}[\cdot]:K_{j}\rightarrow [0,+\infty)$, involving positive Stieltjes measures given by
$$
\alpha^{\rho_1,\rho_2}_{ij}[w]=\int_0^1 w(t)\,dC_{ij}(t),
$$
where $C_{ij}$ is of bounded variation, such that, for $i=1,2$,
\begin{itemize}
\item  $\alpha^{\rho_1,\rho_2}_{ii}[\gamma_{i}]<1$,
\item $H_{i}[u,v]\geq A_i^{\rho_1,\rho_2}+\alpha^{\rho_1,\rho_2}_{i1}[u]+\alpha^{\rho_1,\rho_2}_{i2}[v]$ for $(u,v)\in \partial
V_{\rho_1,\rho_2}\,$,
\item the following inequality holds:
\begin{equation}\label{indos}
 f_{1,(\rho_1,\rho_2)} \Bigl(\frac{ c_{\gamma_i}|| \gamma _{i}||}{1 -\alpha_{ii}^{\rho_1,\rho_2}[\gamma_i]}\int_{a_i}^{b_i}\mathcal K_{i}(s)g_{i}(s)\,ds +\dfrac{1}{M_i}\Bigr)+ \dfrac{c_{\gamma_i}|| \gamma _{i}||A_i^{\rho_1,\rho_2}}{\rho_i(1-\alpha_{ii}^{\rho_1,\rho_2}[\gamma_i])}>1,
\end{equation}{}
where
\begin{align*}
f_{1,(\rho_1,\rho_2)}=&\inf \Bigl\{ \frac{f_1(u,v)}{ \rho_1}:\; (u,v)\in [\rho_1,\rho_1/c_1]\times[-\rho_2/c_2, \rho_2/c_2]\Bigr\},\\
f_{2,(\rho_1\rho_2)}=&\inf \Bigl\{ \frac{f_2(u,v)}{ \rho_2}:\; (u,v)\in[0,\rho_1/c_1]\times[\rho_2, \rho_2/c_2]\Bigr\},\\
 \frac{1}{M_i}:=&\inf_{t\in [a_i,b_i]}\int_{a_i}^{b_i} k_i(t,s) g_i(s)\,ds.
\end{align*}
\end{itemize}
\end{enumerate}
Then $i_{K}(T,V_{\rho_1,\rho_2})=0$.
\end{lem}

\begin{proof}
Let $e(t)\equiv 1$ for $t\in [0,1]$. Then $(1,1)\in K$. We prove that
\begin{equation*}
(u,v)\ne T(u,v)+\lambda (1,1)\quad\text{for } (u,v)\in \partial V_{\rho_1,\rho_2}\quad\text{and } \lambda\geq 0.
\end{equation*}
In fact, if this does not happen, there exist $(u,v)\in \partial V_{\rho_1,\rho_2}$ and $\lambda \geq 0$ such that $(u,v)=T(u,v)+\lambda (1,1)$.
Without loss of generality, we can assume that for all $t\in [a_1,b_1]$ we have
$\rho_1\leq u(t)\leq {\rho_1/c_1}$,  $\min u(t)=\rho_1$ and  $-\rho_2/c_2\leq v(t)\leq {\rho_2/c_2}$.
For $t\in [0,1]$, we have
\begin{align*}
u(t)&=\gamma_1(t)H_1 [u,v]+F_{1}(u,v)(t)+ \lambda,
\end{align*}
thus, applying $\alpha_{11}^{\rho_1,\rho_2}$ to both sides of the equality, we obtain
\begin{align*}
\alpha_{11}^{\rho_1,\rho_2}[u]&=\alpha_{11}^{\rho_1,\rho_2}[\gamma_1]H_1[u,v]+\alpha_{11}^{\rho_1,\rho_2}[F_1(u,v)]+
 \alpha_{11}^{\rho_1,\rho_2}[\lambda]\\
&\geq \alpha_{11}^{\rho_1,\rho_2}[\gamma_1](A_1^{\rho_1,\rho_2}+\alpha_{11}^{\rho_1,\rho_2}[u]+\alpha_{12}^{\rho_1,\rho_2}[v])
+\alpha_{11}^{\rho_1,\rho_2}[F_1(u,v)]\\
&\geq \alpha_{11}^{\rho_1,\rho_2}[\gamma_1](A_1^{\rho_1,\rho_2}+\alpha_{11}^{\rho_1,\rho_2}[u])
+\int_{0}^{1}\mathcal K_{1}(s)g_{1}(s)f_1(u(s),v(s))\,ds.
\end{align*}
We get
\begin{align*}
(1-\alpha_{11}^{\rho_1,\rho_2}[\gamma_1])\alpha_{11}^{\rho_1,\rho_2}[u]
&\geq \alpha_{11}^{\rho_1,\rho_2}[\gamma_1]A_1^{\rho_1,\rho_2}+\int_{0}^{1}\mathcal K_{1}(s)g_{1}(s)f_1(u(s),v(s))\,ds
\end{align*}
and
\begin{align*}
\alpha_{11}^{\rho_1,\rho_2}[u]
&\geq \frac{\alpha_{11}^{\rho_1,\rho_2}[\gamma_1]A_1^{\rho_1,\rho_2}}{1-\alpha_{11}^{\rho_1,\rho_2}[\gamma_1]}
+\frac{1}{1-\alpha_{11}^{\rho_1,\rho_2}[\gamma_1]} \int_{0}^{1}\mathcal K_{1}(s)g_{1}(s)f_1(u(s),v(s))\,ds.
\end{align*}
Consequently, for $t\in [a_1,b_1]$, we have
\begin{align*}
u(t)&\geq \gamma_{1}(t)\left(A_1^{\rho_1,\rho_2}+\alpha_{11}^{\rho_1,\rho_2}[u]+\alpha_{12}^{\rho_1,\rho_2}[v]\right)
+\int_{0}^{1}k_{1}(t,s)g_{1}(s)f_1(u(s),v(s))\,ds+\lambda\\
 &\geq \gamma_{1}(t)A_1^{\rho_1,\rho_2}+\gamma_{1}(t)\alpha_{11}^{\rho_1,\rho_2}[u]+\int_{0}^{1}k_{1}(t,s)g_{1}(s)f_1(u(s),v(s))\,ds+\lambda\\
&\geq \gamma_{1}(t)A_1^{\rho_1,\rho_2}+\gamma_{1}(t)
\frac{\alpha_{11} ^{\rho_1,\rho_2}[\gamma_1]A_1^{\rho_1,\rho_2}}{1-\alpha_{11}^{\rho_1,\rho_2}[\gamma_1]}+
\frac{\gamma_{1}(t)}{1-\alpha_{11}^{\rho_1,\rho_2}[\gamma_1]}\int_{0}^{1}\mathcal K_{1}(s)g_{1}(s)f_1(u(s),v(s))\,ds\\
&\,\,\,\,\,\,+\int_{0}^{1}k_{1}(t,s)g_{1}(s)f_1(u(s),v(s))\,ds+\lambda\\
&\geq \frac{\gamma_1(t)A_1^{\rho_1,\rho_2}}{1-\alpha_{11}^{\rho_1,\rho_2}[\gamma_1]}+
\frac{\gamma_1(t)}{1-\alpha_{11}^{\rho_1,\rho_2}[\gamma_1]}\int_{a_1}^{b_1}\mathcal K_{1}(s)g_{1}(s)f_1(u(s),v(s))\,ds\\
&\,\,\,\,\,\,+\int_{a_1}^{b_1}k_{1}(t,s)g_{1}(s)f_1(u(s),v(s))\,ds+\lambda\\
&=\rho_1\frac{\gamma_1(t)A_1^{\rho_1,\rho_2}}{\rho_1(1-\alpha_{11}^{\rho_1,\rho_2}[\gamma_1])}
+\rho_1\frac{\gamma_1(t)}{1-\alpha_{11}^{\rho_1,\rho_2}[\gamma_1]}\int_{a_1}^{b_1}\mathcal K_{1}(s)g_{1}(s)
\frac{f_1(u(s),v(s))}{\rho_1}\,ds\\
&\,\,\,\,\,\,+\rho_1\,\int_{a_1}^{b_1}k_{1}(t,s)g_{1}(s)\frac{f_1(u(s),v(s))}{\rho_1}\,ds+\lambda
\end{align*}
and therefore, taking the infimum over $[a_1,b_1]$, we obtain
\begin{align*}
\rho_1=&\min_{t \in [a_1,b_1]}u(t)  \geq  \rho_1\frac{c_{\gamma_1}\|\gamma_1\|A_1^{\rho_1,\rho_2}}
{\rho_1(1-\alpha_{11}^{\rho_1,\rho_2}[\gamma_1])}\\
&+\rho_1f_{1,(\rho_1, {\rho_2})}\Bigl(\frac{c_{\gamma_1}\|\gamma_1\|}{1-\alpha_{11}^{\rho_1,\rho_2}[\gamma_1]}
\int_{a_1}^{b_1}\mathcal K_{1}(s)g_{1}(s)\,ds+\inf_{t\in[a_1,b_1]}\int_{a_1}^{b_1}k_{1}(t,s)g_{1}(s)\,ds\Bigr)+\lambda.
\end{align*}
Using the hypothesis \eqref{indos} we obtain $\rho_1>\rho_1 +\lambda $, a contradiction since $\lambda\geq 0$.
\end{proof}

In the following Lemma provides a result of index 0 on $V_{\rho_1,\rho_2}$, by controlling the growth of just one nonlinearity $f_i$ in a larger domain. Nonlinearities with different growths were studied also in~\cite{nupa, gipp-nonlin, gipp-nodea, precup1, precup2, ya1}.

\begin{lem}%\label{idx0b3}
Assume that
\begin{enumerate}
\item[$(\mathrm{I}_{\rho_1,\rho_2 }^{0})^{\circ}$]  there exist $\rho_1,\rho_2 >0$, $A^{\rho_1,\rho_2}_{1},A^{\rho_1,\rho_2}_{2}\geq 0$,  linear functionals $\alpha^{\rho_1,\rho_2}_{ij}[\cdot]:K_{j}\rightarrow [0,+\infty)$, involving positive Stieltjes measures given by
$$
\alpha^{\rho_1,\rho_2}_{ij}[w]=\int_0^1 w(t)\,dC_{ij}(t),
$$
where $C_{ij}$ is of bounded variation, such that for almost one $i=1,2$
\begin{itemize}
\item $\alpha^{\rho_1,\rho_2}_{ii}[\gamma_{i}]<1$,
\item $H_{i}[u,v]\geq A_i^{\rho_1,\rho_2}+\alpha^{\rho_1,\rho_2}_{i1}[u]+\alpha^{\rho_1,\rho_2}_{i2}[v]$ for  $(u,v)\in \partial
V_{\rho_1,\rho_2}\,$,
\item the following inequality holds:
\begin{equation}\label{diamante}
 f_{1,(\rho_1,\rho_2)}^{\circ} \Bigl(\frac{ c_{\gamma_i}|| \gamma _{i}||}{1 -\alpha_{ii}^{\rho_1,\rho_2}[\gamma_i]}\int_{a_i}^{b_i}\mathcal K_{i}(s)g_{i}(s)\,ds +\dfrac{1}{M_i}\Bigr)+ \dfrac{c_{\gamma_i}|| \gamma _{i}||A_i^{\rho_1,\rho_2}}{\rho_i(1-\alpha_{ii}^{\rho_1,\rho_2}[\gamma_i])}>1,
\end{equation}{}
where
\begin{align*}
f^{\circ}_{1,(\rho_1 ,{\rho_2})}=&\inf \Bigl\{\frac{f_1(u,v)}{ \rho_1}:\; (u,v)\in[0,\rho_1/c_1]\times[-\rho_2/c_2, \rho_2/c_2]\Bigr\},\\
f^{\circ}_{2,(\rho_1 ,{\rho_2})}=&\inf \Bigl\{\frac{f_2(u,v)}{ \rho_2}:\; (u,v)\in [0,\rho_1/c_1]\times[0,\rho_2/c_2]\Bigr\}.
\end{align*}
\end{itemize}
\end{enumerate}
Then $i_{K}(T,V_{\rho_1,\rho_2})=0$.
\end{lem}
\begin{proof}
Suppose that the condition~\eqref{diamante} holds for $i=1$. Let $(u,v)\in \partial V_{\rho_1,\rho_2 }$ and $\lambda \geq 0$ such
that $(u,v)=T(u,v)+\lambda (1,1)$. Therefore, for all $t\in [a_1,b_1]$, we have $\min u(t)\leq \rho_1$, $0 \leq u(t)\leq \rho_1/c_1$ and
$-\rho_2/c_2\leq v(t)\leq \rho_2/c_2$. For $t\in [0,1]$, we have
\begin{equation*}
  u(t)=\gamma_{1}(t)H_{1}[u,v]+\int_{0}^{1}k_1(t,s)g_1(s)f_1(u(s),v(s))\,ds+{\lambda}.
\end{equation*}
As in the proof of Lemma \ref{idx0n1}, taking the minimum over $[a_{1},b_{1}]$ gives
\begin{align*}
\rho_1\geq
&\min_{t\in [a_1,b_1]}u(t) \geq  \rho_1\frac{c_{\gamma_1}\|\gamma_1\|A_1^{\rho_1,\rho_2}}{\rho_1(1-\alpha_{11}^{\rho_1,\rho_2}[\gamma_1])}\\
&+\rho_1f_{1,(\rho_1, {\rho_2})}^{\circ}\Bigl(\frac{c_{\gamma_1}\|\gamma_1\|}{1-\alpha_{11}^{\rho_1,\rho_2}[\gamma_1]}\int_{a_1}^{b_1}\mathcal K_{1}(s)g_{1}(s)\,ds+\inf_{t\in[a_1,b_1]}\int_{a_1}^{b_1}k_{1}(t,s)g_{1}(s)\,ds\Bigr)+\lambda.
\end{align*}
Using the hypothesis \eqref{diamante} we obtain $\rho_1>\rho_1+\lambda$, a contradiction.
\end{proof}

\begin{rem}
In the case of $[a_1,b_1]=[a_2,b_2]$  the assumptions on the nonlinearities $f_i$ can be relaxed, for example in condition
$(\mathrm{I}_{\rho _{1},\rho_2}^{0})^{\circ}$ we have to control the growth of $f_i$ in the smaller set ${[0,\rho_1/c_1]\times[0, \rho_2/c_2]}$. We refer to the paper \cite{gipp-nodea} for the statement of similar results.

\end{rem}
We now state  a result regarding the existence of at least one, two or three nontrivial solutions. The proof, that follows by the properties of fixed point index, is omitted. We can to state results for four or more nontrivial solutions by expanding the lists in conditions $(S_{5}),(S_{6})$, see for
example the paper~\cite{kljdeds}.

\begin{thm}\label{mult-sys}
If one of the following conditions holds:
\begin{enumerate}
\item[$(S_{1})$]  For $i=1,2$ there exist $\rho _{i},r _{i}\in (0,\infty )$ with $\rho_{i}/c_i<r _{i}$ such that $(\mathrm{I}_{\rho _{1},\rho_2}^{0})\;\;[\text{or}\;(\mathrm{I}_{\rho _{1},\rho_2}^{0})^{\circ }]$, $(\mathrm{I}_{r _{1},r_2}^{1})$ hold;
\item[$(S_{2})$] For $i=1,2$ there exist $\rho _{i},r _{i}\in (0,\infty )$ with $\rho_{i}<r _{i}$ such that $(\mathrm{I}_{\rho _{1},\rho_2}^{1}),\;\;(\mathrm{I}_{r _{1},r_2}^{0})$ hold;
\end{enumerate}
then the system \eqref{syst} has at least one nontrivial solution in $K$.\\
 If one of the following conditions holds:
\begin{enumerate}
\item[$(S_{3})$] For $i=1,2$ there exist $\rho _{i},r _{i},s_i\in (0,\infty )$ with $\rho _{i}/c_i<r_i <s _{i}$ such that $(\mathrm{I}_{\rho_{1},\rho_2}^{0})$, $[\text{or}\;(\mathrm{I}_{\rho _{1},\rho_2}^{0})^{\circ }],\;\;(\mathrm{I}_{r _{1},r_2}^{1})$ $\text{and}\;\;(\mathrm{I}_{s _{1},s_2}^{0})$ hold;
\item[$(S_{4})$] For $i=1,2$ there exist $\rho _{i},r _{i},s_i\in (0,\infty )$ with $\rho _{i}<r _{i}$ and $r _{i}/c_i<s _{i}$ such that $(\mathrm{I}
_{\rho _{1},\rho_2}^{1}),\;\;(\mathrm{I}_{r _{1},r_2}^{0})$ $\text{and}\;\;(\mathrm{I}_{s _{1},s_2}^{1})$ hold;
\end{enumerate}
then the system \eqref{syst} has at least two nontrivial solution in $K$.\\
If one of the following conditions holds:
\begin{enumerate}
\item[$(S_{5})$] For $i=1,2$ there exist $\rho _{i},r _{i},s_i,\sigma_i\in (0,\infty )$ with $\rho _{i}/c_i<r _{i}<s _{i}$ and $s _{i}/c_i<\sigma_{i}$ such that $(\mathrm{I}_{\rho _{1},\rho_2}^{0})\;\;[\text{or}\;(\mathrm{I}_{\rho _{1},\rho_2}^{0})^{\circ }],$ $(\mathrm{I}_{r _{1},r_2}^{1}),\;\;
(\mathrm{I}_{s_1,s_2}^{0})\;\;\text{and}\;\;(\mathrm{I}_{\sigma _{1},\sigma_2}^{1})$ hold;
\item[$(S_{6})$] For $i=1,2$ there exist $\rho _{i},r _{i},s_i,\sigma_i\in(0,\infty )$ with $\rho _{i}<r _{i}$ and $r _{i}/c_i<s _{i}<\sigma _{i}$
such that $(\mathrm{I}_{\rho _{1},\rho_2}^{1}),\;\;(\mathrm{I}_{r_{1},r_2}^{0}),\;\;(\mathrm{I}_{s _{1},s_2}^{1})$ $\text{and}\;\;
(\mathrm{I}_{\sigma _{1},\sigma_2}^{0})$ hold;
\end{enumerate}
then the system \eqref{syst} has at least three nontrivial solution in $K$.
\end{thm}

The results of this Subsection, for example Theorem~\ref{mult-sys}, can be applied to the integral system~\eqref{syst}, yielding results for the elliptic system \eqref{ellbvpsec}. Similar statements were given in~\cite{lan-lin-na, lanwebb} in the case of annular domains.

In the following example we illustrate the applicability of Theorem~\ref{mult-sys}.
\begin{ex}
In $\mathbb{R}^3$ consider  the elliptic system
\begin{equation}\label{ellbvpex}
\begin{cases}
\Delta u + |x|^{-4}\left[\frac{3}{10}(u^3+|v|^3)+\frac{1}{2}\right]=0,\,\, |x|\in [1,+\infty), \\
\Delta v + |x|^{-4}(u^{\frac{1}{2}}+v^2+1)=0,\,\, |x|\in [1,+\infty),\\
u(x)=2 u(4 x)\,\,\mbox{ for }{x\in \partial B_{1}},\,\, \displaystyle \lim_{|x|\to+\infty}u(|x|)=\frac{1}{10} \sqrt{u(3x)}+\frac{1}{10} v^3\left(\frac{7}{2}x\right),\\
v(x)=\displaystyle -\frac{4}{3} \frac{\partial v}{\partial r}(2 x)\,\,\mbox{ for }{x\in \partial B_{1}},\,\,\displaystyle \lim_{|x|\to+\infty}v(|x|)= \frac{1}{5}\sqrt{u\left(\frac{7}{3}x\right)}+\frac{1}{10} v^2\left(\frac{5}{2}x\right).
\end{cases}
\end{equation}

To the system~\eqref{ellbvpex} we associate the system of second order ODEs
\begin{equation*}
\begin{cases}
u''(t)+ \frac{3}{10}\left(u^3(t)+|v|^3(t)\right)+\frac{1}{2}=0,\,\, t\in [0,1], \\
v''(t) +u^{\frac{1}{2}}(t)+v^2(t)+1=0,\,\, t\in [0,1],\\
u(0)=\frac{1}{10} \sqrt{u\left(\frac{1}{3}\right)}+\frac{1}{10} v^3\left(\frac{2}{7}\right),\,\,\,\,\,\; u(1)=2u\left(\frac{1}{4}\right),\\
v(0)=\frac{1}{5}\sqrt{u\left(\frac{3}{7}\right)}+\frac{1}{10} v^2\left(\frac{2}{5}\right),\;\,\,\,\,\, v(1)=\frac{1}{3}v'\left(\frac{1}{2}\right).%
\end{cases}
\end{equation*}
By direct computation, we have
\begin{equation*}
 \frac{1}{m_{1}}=\sup_{t\in [0,1]}\int_{0}^{1} k_{1}(t,s)\,ds=\frac{49}{128}\,\,\,\,\, \mbox{  and  }\,\,\,\,\, \frac{1}{m_{2}}=\sup_{t\in [0,1]}\int_{0}^{1} |k_{2}(t,s)|\,ds=\frac{1}{8}.
 \end{equation*}{}
We fix $[a_1,b_1]=[a_2,b_2]=[\frac{1}{4},\frac{1}{2}]$, obtaining $c_1=\frac{1}{32},\, \, c_2=\frac{1}{4}$,
$$
 \frac{1}{M_{1}} = \inf_{t\in
[1/4,1/2]}\int_{1/4}^{1/2} k_1(t,s)\,ds=\frac{3}{16}
$$
and
$$
 \frac{1}{M_{2}} = \inf_{t\in
[1/4,1/2]}\int_{1/4}^{1/2} k_2(t,s)\,ds=\frac{3}{32}.
$$
With the choice of
\begin{align*}
&\rho_1=\frac{1}{16},\,\,\,\rho_2=\frac{1}{32},\,\,\,\alpha^{\rho_1,\rho_2}_{11}[u]=\alpha^{\rho_1,\rho_2}_{21}[u]=0,\\
&r_1=2.01,\,\,\, r_2=1,\,\,\, A_1^{r_1,r_2}=\frac{1}{10}\sqrt{2.01},\,\,\,\alpha^{r_1,r_2}_{12}[v]=\frac{1}{10}\,
v\left(\frac{2}{7}\right),\,\,\,\\
&\,\,\,\,\,\,\,\,\,\,\,\,\,\,\,\,\,\,\,\,\,\,\,\,\,\,\,\,\,\,\,\,\,\,\,\,\,\,\,\,\,\,\,\,\,\,\,\,\, A_2^{r_1,r_2}=\frac{1}{5}\sqrt{2.01},\,\,\,\alpha^{r_1,r_2}_{22}[v]=\frac{1}{10}\,
v\left(\frac{2}{5}\right),\\
&s_1=5,\,\,\,s_2=11,\,\, \,\alpha^{s_1,s_2}_{11}[u]=\frac{1}{130}\,
u\left(\frac{1}{3}\right),\,\,\, \alpha^{s_1,s_2}_{21}[u]=\frac{1}{100}\, u\left(\frac{3}{7}\right),\,\,\, \alpha^{s_1,s_2}_{22}[v]=\frac{1}{100} v\left(\frac{2}{5}\right),
\end{align*}
we obtain
\begin{align*}
&\alpha^{r_1,r_2}_{22}[\gamma_{2}]<1,\,\,\,\,\alpha^{s_1,s_2}_{11}[\gamma_{1}]<1,\,\,\,\,\alpha^{s_1,s_2}_{22}[\gamma_{2}]<1,\\
&H_{1}[u,v]\leq\, A_1^{r_1,r_2}+\alpha^{r_1,r_2}_{12}[v],\,\,\,H_{2}[u,v]\leq\, A_2^{r_1,r_2}+\alpha^{r_1,r_2}_{22}[v],\,\,\,(u,v)\in [0,r_1]\times[0,r_2],\,\\
&H_{1}[u,v]\geq
\alpha^{s_1,s_2}_{11}[u],\,\,\,
(u,v)\in [s_1,s_1/c_1]\times[0,s_2/c_2],\,\\
&H_{2}[u,v]\geq
\alpha^{s_1,s_2}_{21}[u]+\alpha^{s_1,s_2}_{22}[v],\,\,\,
(u,v)\in [0,s_1/c_1]\times[s_2,s_2/c_2],\,\\
\inf & \Bigl\{ f_2(u,v):\; (u,v)\in [0,32\rho_1]\times[0,4\rho_2]
\Bigr\}= f_2(0,0)=1>0.33, \\
\sup &\Bigl\{ f_1(u,v):\; (u,v)\in [0, r_1]\times[-r_2,
r_2]\Bigr\}=f_1\left(2.01,1\right)=3.236
<3.878   , \\
\sup &\Bigl\{ f_2(u,v):\; (u,v)\in [0, r_1]\times[-r_2,
r_2]\Bigr\}=f_2\left(2.01,1\right)=3.417
<5.125, \\
\inf &\Bigl\{ f_1(u,v):\; (u,v)\in
[s_1,32s_1]\times[0,4s_2]\Bigr\}=f_1(5,0)=38
>26.557, \\
\inf &\Bigl\{ f_2(u,v):\; (u,v)\in [0,32s_1]\times[s_2,
4s_2]\Bigr\}=f_2(0,11)=122 >117.075.
\end{align*}
Thus the conditions $(\mathrm{I}^{0}_{\rho_{1},\rho_2})^{\circ}$, $(\mathrm{I}%
^{1}_{r_1,r_{2}})$ and $(\mathrm{I}^{0}_{s_1,s_{2}})$ are satisfied and, from Theorem~\ref{mult-sys},
the system~\eqref{ellbvpex} has at least two nontrivial solutions.
\end{ex}
\subsection{Non-existence results for perturbed integral systems}
We now present some non-existence results for the integral system~\eqref{syst}. We begin with a Theorem where the nonlinearities and the functionals are enough ``small''.
\begin{thm}\label{nonexi1s}
Assume that, for $i=1,2$, there exist $A_i>0$ and $\lambda_i>0$ such that for $f_i:[0,+\infty)\times \mathbb{R}\to [0,+\infty)$ one has
\begin{itemize}
\item  $H_{i}[u_1,u_2]\leq A_i\|u_i\|$ for  $(u_1,u_2)\in K$,
\item $f_i(z_1,z_2)\leq\lambda_{i}m_i|z_i|$ for $(z_1,z_2)\in [0,+\infty)\times \mathbb{R}$,
\item $\|\gamma_i\|A_i+\lambda_i<1$.
\end{itemize}
Then there is no nontrivial solution of the system~\eqref{syst} in $K$.
\end{thm}
\begin{proof}
Suppose that there exists  $(u,v)\in K$ such that $ (u,v)=T(u,v)$ and  assume that $|| v|| =\nu>0$.
Then we have, for $t\in[0,1]$,
\begin{align*}
 |v(t)|&=|\gamma_{2}(t)H_{2}[u,v]+F_{2}(u,v)(t)|\\
 &\leq \|\gamma_{2}\|H_{2} [u,v]+\int_0^1|k_{2}(t,s)|\,g_2(s)\,f_2(u(s),v(s))ds\\
  &\leq \|\gamma_{2}\|A_2\|v\|+\lambda_2m_2\int_0^1|k_{2}(t,s)|\,g_2(s)|v(s)|ds\\
  &\leq \|\gamma_{2}\|A_2\nu+\lambda_2m_2\nu \int_0^1|k_{2}(t,s)|\,g_2(s)ds.
\end{align*}
Taking the supremum over $[0,1]$ gives
\begin{align*}
\nu\leq &\nu\Bigl(\|\gamma_{2}\|A_2+\lambda_2 m_2\sup_{t \in [0,1]}\int_0^1|k_{2}(t,s)|\,g_2(s)ds\Bigr).
\end{align*}
We obtain $\nu <\nu$, a contradiction that proves the result.

The case $|| u|| =\nu>0$ is simpler and we omit the proof.
\end{proof}
In the next Theorem the nonlinearities and the functionals are enough ``large''.
\begin{thm}\label{nonexi2s}
Assume that, for $i=1,2$, there exist $A_i>0$ and $\lambda_i>0$ such that for $f_i:[0,+\infty)\times \mathbb{R}\to [0,+\infty)$ one has
\begin{itemize}
\item   $H_{i}[u_1,u_2]\geq A_i\|u_i\|$ for $(u_1,u_2)\in K$,
\item $f_i(z_1,z_2)\geq \lambda_{i}M_i\,z_i$ for $(z_1,z_2)\in [0,+\infty)\times \mathbb{R}$,
\item $c_{\gamma_i}\|\gamma_i\|A_i+\lambda_i>1$.
\end{itemize}
Then there is no nontrivial solution of the system \eqref{syst} in $K$.
\end{thm}

\begin{proof}
Suppose that there exists  $(u,v)\in K$ such that $ (u,v)=T(u,v)$ with $(u,v)\neq 0$ and $\displaystyle \min_{t\in [a_1,b_1]}u(t)=\theta>0$ (the other case is similar). Then we have, for $t\in[a_1,b_1]$,
\begin{align*}
 u(t)&=\gamma_{1}(t)H_{1 }[u,v]+F_{1}(u,v)(t)\\
 &\geq \gamma_{1}(t)H_{1} [u,v]+\int_{a_1}^{b_1}k_{1}(t,s)g_1(s)f_1(u(s),v(s))ds\\
  &\geq c_{\gamma_1} \|\gamma_{1}\|A_1\|u\|+\lambda_1M_1\int_{a_1}^{b_1}k_{1}(t,s)g_1(s)u(s)ds\\
  &\geq c_{\gamma_1}  \|\gamma_{1}\|A_1\theta+\lambda_1M_1\theta \int_{a_1}^{b_1}k_{1}(t,s)g_1(s)ds.
\end{align*}
Taking the infimum for $t\in [a_1,b_1]$ gives
\begin{align*}
\theta\geq &\theta\Bigl(c_{\gamma_1} \|\gamma_{1}\|A_1+\lambda_1 M_1\inf_{t \in [a_1,b_1]}\int_{a_1}^{b_1}k_{1}(t,s)g_1(s)ds\Bigr).
\end{align*}
We obtain $\theta >\theta$, a contradiction that proves the result.
\end{proof}
In the last non-existence result the contribution of the nonlinearities and the functionals is of a ``mixed'' type.
\begin{thm}\label{noexis3}
Assume that, for $i=1,2$, there exist $A_i,\lambda_i>0$ such that the assumptions in Theorem\,\ref{nonexi1s} are verified for an $i\in\{1,2\}$ and the assumptions in Theorem~\ref{nonexi2s} are verified for  the other index. Then there is no nontrivial solution of the system~\eqref{syst} in $K$.
\end{thm}
\begin{proof}
Assume, on the contrary, that there exists a nontrivial $(u,v)\in K$ such that $(u,v)=T(u,v)$ with, for example, $\|u\|\neq 0$. Then
the functions $\gamma_{1}$, $H_{1}$ and $f_1$ satisfy either the assumptions in Theorem\,\ref{nonexi1s} or the assumptions in
Theorem~\ref{nonexi2s} and the proof follows consequently.
\end{proof}
We present an example that illustrates the applicability of Theorem~\ref{noexis3}.
\begin{ex}
Fix in the system~\eqref{syst} $g_1\equiv 1$, $\beta_1=2$ and $\eta=\frac{1}{4}$. By direct calculations, we obtain $m_1=\displaystyle\frac{128}{49}$. Thus the functions
\begin{equation*}
f_1(u,v)=
\begin{cases}
\dfrac{1}{4}u^2\, \arctan\, v^2, &  u\leq 1,\\
\\
 \dfrac{1}{4}u^{2/3}\, \arctan\, v^2,& u>1,
\end{cases}
\,\,\,\mbox{ and }\,\,\,\, H_1[u,v]=\frac{1}{6} u(1/2)\,\left |\sin(v(1/3))\right|
\end{equation*}
satisfy the conditions of Theorem\,\ref{nonexi1s} with $A_1=\frac{1}{6}$ and $\lambda_1=\frac{1}{5}$.

Now fix $g_2\equiv 1$, $\beta_2=\frac{1}{3}$, $\xi=\frac{1}{2}$ and $[a_2,b_2]=\left[\frac{1}{4},\frac{1}{2}\right]$. In this case we obtain $M_2=\displaystyle \frac{32}{3}$. Thus the functions
\begin{equation*}
f_2(u,v)=
\begin{cases}
\dfrac{32}{3}\,(2+\cos u)\, v^{2/3}, &  v\leq 1,\\
\\
 \dfrac{32}{3}\,(2+\cos u)\, v^{2},& v>1,
\end{cases}
\,\,\,\mbox{ and }\,\,\,\, H_2[u,v]=e^{u(1/2)}\,v(1/3)
\end{equation*}
satisfy the conditions of Theorem\,\ref{nonexi2s} with $A_2=\frac{1}{6}$ and $\lambda_2=1$.

With these choices, by Theorem~\ref{noexis3}, the system~\eqref{syst} has no solutions in $K$.
\end{ex}

\section{Perturbing a system with Dirichlet BCs}\label{three}
In~\cite{do6} the authors raised the  question of the existence of multiple positive solutions of the elliptic system~\eqref{do} under more general BCs.

Using our theory we can deal with the system of elliptic equations
\begin{equation}\label{dsyst}
\begin{cases}
\Delta u + h_1(|x|)f_1(u,v)=0,\ |x|\in [1,+\infty), \\
\Delta v +h_2(|x|)f_2(u,v)=0,\, |x|\in [1,+\infty),\\
\end{cases}
\end{equation}
subject to the BCs
\begin{equation}\label{dsystBC1}
\begin{cases}
u(x)=H_1[u,v]\,\,\text{for  }x\in \partial B_{1},\, \,\,\,\,\displaystyle \lim_{|x|\to+\infty}u(|x|)=0, \\
v(x)=H_2[u,v]\,\,\text{for  }x\in \partial B_{1},\,\,\,\,\, \displaystyle\lim_{|x|\to+\infty}v(|x|)=0,
\end{cases}
\end{equation}
or
\begin{equation}\label{dsystBC2}
\begin{cases}
u(x)=0\,\,\text{for  }x\in \partial B_{1},\, \,\,\,\,\displaystyle \lim_{|x|\to+\infty}u(|x|)=H_1[u,v], \\
v(x)=0\,\,\text{for  }x\in \partial B_{1},\,\,\,\,\, \displaystyle\lim_{|x|\to+\infty}v(|x|)=H_2[u,v],
\end{cases}
\end{equation}
or
\begin{equation}\label{dsystBC3}
\begin{cases}
u(x)=H_1[u,v]\,\,\text{for  }x\in \partial B_{1},\, \,\,\,\,\displaystyle \lim_{|x|\to+\infty}u(|x|)=0, \\
v(x)=0\,\,\text{for  }x\in \partial B_{1},\,\,\,\,\, \displaystyle\lim_{|x|\to+\infty}v(|x|)=H_2[u,v].
\end{cases}
\end{equation}

We examine, for example,   the perturbed integral system
\begin{equation}\label{syst2}
\begin{cases}
u(t)=t\, H_{1}[u,v]+\int_{0}^{1}k(t,s)g_1(s)f_1(u(s),v(s))\,ds, \\
v(t)=t\, H_{2}[u,v]+\int_{0}^{1}k(t,s)g_2(s)f_2(u(s),v(s))\,ds,%
\end{cases}
\end{equation}
associated to the system~\eqref{dsyst} with the BCs~\eqref{dsystBC1}.

In this case, the Green's function $k$ is given by
\begin{equation}\label{kdir}
k(t,s):=
\begin{cases}
 s(1-t), &s\le t, \\
  t(1-s),&s>t.
\end{cases}
\end{equation}

In $[0,1]\times  [0,1]$ the kernel $k$ is continuous and non-negative and has been studied, for example, in \cite{jw-gi-jlms}, where it was shown that
\begin{align*}
k(t,s)\leq \Phi(s) \text{ for } &t,\,s \in [0,1], \\
k(t,s) \geq {c}_{k}\Phi(s) \text{ for } &t\in [a,b] \text{ and  } \, s \in [0,1],
\end{align*}%
where
$\Phi(s)= s(1-s)$,  $[a,b]$ is an arbitrary subset of $(0,1)$ and ${c}_{k}=\min\{a, 1-b\}$.
Set $\gamma_1(t)=\gamma_2(t)=:\gamma(t)=t$;
we have $\gamma \in C[0,1]$ and
\begin{equation*}
\gamma (t)\geq {c}_{\gamma}\| \gamma\|\;\text{ for
}\;t\in [a,b],
\end{equation*}%
where
$||\gamma||=1$ and ${c}_{\gamma}=a$.

Due to the properties above, we are able to work in the cone of non-negative functions $K:=\tilde{K}\times \tilde{K}$ in $C[0,1]\times C[0,1]$,
where
\begin{equation*}
\tilde{K}=\{w\in C[0,1]:\,w\geq 0,\,\,\min_{t\in [a,b]}w(t)\geq {c}\|
w\|\},
\end{equation*}%
with  $c=\min \{{c}_{k},{c}_{\gamma}\}$.
Our theory works in this case with suitable changes.

For example (with abuse of notation), in condition~\eqref{ind1s} of Lemma~\ref{ind1L} we would consider
\begin{align*}
f_{i}^{\rho_1,\rho_2}:=\sup \Bigl\{\frac{f_{i}(u,v)}{\rho_i}:\;(u,v)\in [ 0,\rho_1 ]\times [ 0,\rho_2 ]\Bigr\},\,\,\,
\frac{1}{m_{i}}:=\sup_{t\in \lbrack 0,1]}\int_{0}^{1}k(t,s)g_{i}(s)\,ds,
\end{align*}
and in condition \eqref{indos} of Lemma~\ref{idx0n1} we would have
 \begin{equation*}
 f_{1,(\rho_1,\rho_2)}=\inf \Bigl\{ \frac{f_1(u,v)}{ \rho_1}:\; (u,v)\in [\rho_1,\rho_1/c]\times[0, \rho_2/c]\Bigr\}.
\end{equation*}

In the next example we briefly illustrate the constants involved in our theory.
\begin{ex}
Consider  in $\mathbb{R}^3$  the elliptic system
\begin{equation}\label{ellbvpdir}
\begin{cases}
\Delta u + |x|^{-4}\left(u^3+v^2+\frac{1}{2}\right)=0,\,\, |x|\in [1,+\infty), \\
\Delta v + |x|^{-4}\left(\displaystyle\frac{\sqrt{u}}{2}+v^2\right)=0,\,\, |x|\in [1,+\infty),\\
u(x)=\dfrac{1}{10}+\displaystyle\frac{\sqrt{v(2x)}}{2 \sqrt{5}}\,\,\,\mbox{ for }{x\in \partial B_{1}},\,\, \displaystyle \lim_{|x|\to+\infty}u(|x|)=0,\\
v(x)=\dfrac{1}{10}+\displaystyle\frac{u^2(3x)}{20}\,\,\,\,\,\mbox{ for }{x\in \partial B_{1}},\,\,\displaystyle \lim_{|x|\to+\infty}v(|x|)= 0.
\end{cases}
\end{equation}

To the system~\eqref{ellbvpdir} we associate the system of second order ODEs
\begin{equation*}
\begin{cases}
u''(t)+ u^3(t)+v^2(t)+\frac{1}{2}=0,\,\, t\in [0,1], \\
v''(t) +\displaystyle\frac{\sqrt{u(t)}}{2}+v^2(t)=0,\,\, t\in [0,1],\\
u(0)=0,\,\,\,\,\,\; u(1)=\dfrac{1}{10}+\displaystyle\frac{\sqrt{v(1/2)}}{2 \sqrt{5}},\\
v(0)=0,\;\,\,\,\,\, v(1)=\dfrac{1}{10}+\displaystyle\frac{u^2(1/3)}{20},%
\end{cases}
\end{equation*}
and the perturbed integral system
\begin{equation}\label{systdir}
\begin{cases}
u(t)=t\, H_{1}[u,v]+\int_{0}^{1}k(t,s)f_1(u(s),v(s))\,ds, \\
v(t)=t\, H_{2}[u,v]+\int_{0}^{1}k(t,s)f_2(u(s),v(s))\,ds,%
\end{cases}
\end{equation}
with
$H_1[u,v]=\displaystyle\frac{1}{10}+\frac{\sqrt{v(1/2)}}{2
\sqrt{5}}$ and
$H_2[u,v]=\displaystyle\frac{1}{10}+\frac{u^2(1/3)}{20}$.\\ By
direct computation, we have
\begin{equation*}
 \frac{1}{m}=\sup_{t\in [0,1]}\int_{0}^{1} k(t,s)\,ds=\frac{1}{8}.
 \end{equation*}{}
We fix $[a,b]=[\frac{1}{4},\frac{3}{4}]$, obtaining $c=\frac{1}{4}$,
$$
 \frac{1}{M}: = \inf_{t\in
[1/4,3/4]}\int_{1/4}^{3/4} k(t,s)\,ds=\frac{1}{16}.
$$
With the choice of
\begin{align*}
&\rho_1=\frac{1}{39},\,\,\,\rho_2=\frac{1}{10},\,\,\,A_1^{\rho_1,\rho_2}=\frac{1}{10},\,\,\,\,\alpha^{\rho_1,\rho_2}_{12}[v]=\frac{1}{2\sqrt{5}}v(1/2),\\
&r_1=2,\,\,\, r_2=2,\,\,\, A_1^{r_1,r_2}=\frac{1}{10}+\frac{1}{2\sqrt{5}}\sqrt{2},\,\,\, A_2^{r_1,r_2}=\frac{1}{10},\,\,\,\,\alpha^{r_1,r_2}_{21}[u]=\frac{1}{10}u(1/3),\,\\
&s_1=5,\,\,\,s_2=16,\,\,
\,A_1^{s_1,s_2}=\frac{1}{10},\,\,\,\,\alpha^{s_1,s_2}_{12}[v]=\frac{1}{16\sqrt{5}}v(1/2),\,
\,\,\, A_2^{s_1,s_2}=\frac{1}{10},
\end{align*}
one can verify that
\begin{align*}
&H_{1}[u,v]\geq\, A_1^{\rho_1,\rho_2}+\alpha^{\rho_1,\rho_2}_{12}[v],\,\,\,(u,v)\in [0,4\rho_1]\times[0,4\rho_2],\,\\
&H_{1}[u,v]\leq\, A_1^{r_1,r_2},\,\,\,H_{2}[u,v]\leq\, A_2^{r_1,r_2}+\alpha^{r_1,r_2}_{21}[u],\,\,\,(u,v)\in [0,r_1]\times[0,r_2],\,\\
&H_{1}[u,v]\geq A_1^{s_1,s_2}+
\alpha^{s_1,s_2}_{12}[v],\,\,\,
(u,v)\in [s_1,4s_1]\times[0,4s_2],\,\\
&H_{2}[u,v]\geq A_2^{s_1,s_2},\,\,\,(u,v)\in [0,4s_1]\times[s_2,4s_2],\,\\
\inf & \Bigl\{ f_1(u,v):\; (u,v)\in [0,4\rho_1]\times[0,4\rho_2]
\Bigr\}= f_2(0,0)=0.5>0.01, \\
\sup &\Bigl\{ f_1(u,v):\; (u,v)\in [0, r_1]\times[0,
r_2]\Bigr\}=f_1(2,2)=12.5
<12.67   , \\
\sup &\Bigl\{ f_2(u,v):\; (u,v)\in [0, r_1]\times[0,
r_2]\Bigr\}=f_2(2,2)=4.70
<13.6, \\
\inf &\Bigl\{ f_1(u,v):\; (u,v)\in
[s_1,4s_1]\times[0,4s_2]\Bigr\}=f_1(5,0)=125.5
>78, \\
\inf &\Bigl\{ f_2(u,v):\; (u,v)\in [0,4s_1]\times[s_2,
4s_2]\Bigr\}=f_2(0,16)=256 >171.6.
\end{align*}
It follows that the conditions $(\mathrm{I}^{0}_{\rho_{1},\rho_2})^{\circ}$, $(\mathrm{I}%
^{1}_{r_1,r_{2}})$ and $(\mathrm{I}^{0}_{s_1,s_{2}})$ are
satisfied and therefore the system~\eqref{systdir} has at least two
positive solutions.
\end{ex}
\section*{Acknowledgments}
G. Infante and P. Pietramala were partially supported by G.N.A.M.P.A. - INdAM (Italy). F. Cianciaruso is a member of  G.N.A.M.P.A.

\end{document}